\documentclass{amsart}

\usepackage{amssymb,amsfonts,latexsym,amsmath,amsthm,graphicx}

\numberwithin{equation}{section}
\newtheorem{thm}{Theorem}[section]

\newtheorem{cor}[thm]{Corollary}

\newtheorem{remark}{Remark}

\newtheorem*{thmlambda}{Theorem \ref{thm:lambda}}
\newtheorem*{thmtorsion}{Theorem \ref{thm:torsion}}

\newcommand{\CC}{\mathbb{C}}

\newcommand{\DD}{\mathbb{D}}

\newcommand{\ol}{\overline}

\newcommand{\la}{\langle}
\newcommand{\ra}{\rangle}
\newcommand{\e}{\varepsilon}

\newcommand{\p}{\partial}

\input colordvi

\begin{document}

\title[Published in Mathematical Proceedings of the Royal Irish Academy]{Self-commutators of Toeplitz operators and isoperimetric inequalities}

\author[S. R. Bell]{Steven R. Bell}
\email{bell@math.purdue.edu}

\author[T. Ferguson]{Timothy Ferguson}
\email{tjferguson1@ua.edu}

\author[E. Lundberg]{Erik Lundberg}
\email{elundber@fau.edu}


\thanks{Citation for Published Version: S. Bell, T. Ferguson and E. Lundberg, Self-commutators of Toeplitz operators and
isoperimetric inequalities, Mathematical Proceedings of the Royal Irish Academy 114A (2014),
115--132; doi:10.3318/PRIA.2014.114.03 \vspace{1em} \\
This version differs slightly in apperance from the published one.}
\begin{abstract}

For a hyponormal operator, C. R. Putnam's inequality gives an upper bound on the norm of its self-commutator.
In the special case of a Toeplitz operator with analytic symbol in the Smirnov space of a domain, 
there is also a geometric lower bound shown by D. Khavinson (1985) that when combined with Putnam's inequality implies the classical isoperimetric inequality.
For a nontrivial domain, we compare these estimates to exact results.
Then we consider such operators acting on the Bergman space of a domain, and we obtain lower bounds
that also reflect the geometry of the domain.
When combined with Putnam's inequality
they give rise to the Faber-Krahn inequality for the fundamental frequency of a domain
and the Saint-Venant inequality for the torsional rigidity (but with non-sharp constants).
We conjecture an improved version of Putnam's inequality within this restricted setting.
\end{abstract}

\maketitle

\section{Introduction}

Let $H$ be a Hilbert space and $T: H \rightarrow H$ a bounded linear operator.
Let $sp(T)$ denote the spectrum of $T$ defined as $sp(T) := \{\mu \in \CC : T - \mu I \text{ is not invertible} \}$,
where $I$ denotes the identity operator.
If $T$ is normal, then its self-commutator $[T^*,T] := T^*T - T T^*$ vanishes.
If $T$ is not normal, then $|| [T^*,T] ||$ is a measure of ``how far'' $T$ is from normal.

If $[T^*,T] \geq 0$, meaning $\langle [T^*,T] x , x \rangle \geq 0$ for every $x \in H$, then Putnam's inequality \cite{Putnam70} states

\begin{equation}\label{eq:putnam}
 || [T^*,T] || \leq \frac{\text{Area}(sp(T))}{\pi}.
\end{equation}

For instance, suppose $G$ is a bounded domain in the plane with piecewise-smooth boundary, 
and $T_{\phi}:E^2(G) \rightarrow E^2(G)$ is the Toeplitz operator with symbol $\phi$ analytic in a neighborhood of $\ol{G}$,
and $E^2(G)$ is the Smirnov space (see \cite[Chapter 10]{Duren70}).
Then $[T_{\phi}^*,T_{\phi}]$ is positive so that Putnam's inequality applies, and by the spectral mapping theorem \cite{Douglas72},
$sp(T_{\phi}) = \overline{\phi(G)}$, so
$$ || [T_{\phi}^*,T_{\phi}] || \leq \frac{\text{Area} \left( \phi(G) \right)}{\pi} . $$

Within this restricted setting a lower bound was also provided by D. Khavinson \cite{Khav}:
\begin{equation}\label{eq:khav}
|| [T_{\phi}^*,T_{\phi}] || \geq \frac{ 4 \text{Area}^2 \left( \phi(G) \right)}{||\phi'||_{2}^2 \cdot P(G)},
\end{equation}
where $P(G)$ denotes the perimeter of $G$ and $||\cdot||_{2}$ is the norm in the space $E^2(G)$.

A charming consequence of the above follows from setting $\phi(z) = z$ (so we have $||\phi'||_2^2 = ||1||_2^2 = P(G)$), 
and combining Putnam's inequality (\ref{eq:putnam}) with the lower bound (\ref{eq:khav}).
Namely, we obtain
$$ P^2(G) \geq 4\pi \text{Area}(G),$$
which is the classical isoperimetric inequality with a sharp constant.

The current study was inspired by the strong interaction between functional and geometric analysis that is apparent
in the statement of the above results. 

In Section \ref{sec:dqd}, we will revisit the above results in a computational setting (again in the context of Smirnov space).
We take $G$ to be a quadrature domain.
This choice leads to a finite rank commutator
whose matrix can be calculated explicitly.
Within a one parameter family of such examples, 
we compare the estimates mentioned above to the exact value of the norm of the commutator.

In Section \ref{sec:berg}, we prove our main results 
which pertain to Toeplitz operators with analytic symbol but in the setting of the Bergman space of $G$.
We provide a lower bound for the norm of the commutator.
Again the geometry of the spectrum of $T$ comes into play, 
but in addition to its area, the bound depends on the fundamental frequency (i.e., the first Dirichlet eigenvalue of the Laplacian).
Let $G_\phi$ denote the Riemann surface formed by sheets over $\phi(G)$. 
Then we prove the following.

\newcommand{\thmlambdatext}{
Suppose that $G$ is a bounded domain in the plane, and $\phi$ is analytic and locally univalent in $G$.
Consider the Bergman space $A^2(G)$.
Let $T_{\phi}:A^2(G) \rightarrow A^2(G)$ be the Toeplitz operator $T_{\phi}(f) = \phi \cdot f$ with symbol $\phi$.
Then we have the lower bound
\[
||[T_{\phi}^*,T_{\phi}] || \ge 
\sup_{\substack{W \subset G_\phi \\ \phi^{-1}(W) \text{\rm\  piecewise $C^1$}}} 
\frac{4^2 \pi}{\lambda_W^2 \text{Area}(G_\phi)},
\]
where $\lambda_W$ is the first eigenvalue of the Laplacian on $W$,
and $G_{\phi}$ is the Riemann surface for $\phi$ formed by sheets over $\phi(G)$.
In particular, if $G$ has smooth boundary, we have 
\[
||[T_{\phi}^*,T_{\phi}] || \ge 
\frac{4^2 \pi}{\lambda_{G_\phi}^2 \text{Area}(G_\phi)}.
\]
}

\begin{thm}\label{thm:lambda}
\thmlambdatext
\end{thm}

Consider the case $\phi = z$, and assume $G$ has smooth boundary.
Taking $W=G$ and combining the resulting inequality in the theorem with Putnam's inequality (\ref{eq:putnam}),
we obtain
\[
\frac{\text{Area}(G)}{\pi} \ge
||[T_z^*,T_z] || \ge 
\frac{4^2 \pi}{\lambda^2 \text{Area}(G)},
\]
where $\lambda$ is the first eigenvalue for $G$.

In particular, this gives a bound on $\lambda$:

$$\lambda \ge \frac{4 \pi}{\text{Area}(G)}.$$

Except for the constant, which is not sharp, this is the inequality associated with the Faber-Krahn Theorem
which states that, with fixed area, the drum with the lowest base tone is given by a disc.
The precise form is inequality (19) in \cite{Osserman78}:

$$\lambda \ge \frac{j^2 \pi}{\text{Area}(G)},$$
where $j \approx 2.4$ is the first positive zero of the Bessel function $J_0$.

\newcommand{\thmtorsiontext}{
Suppose that $G$ is a bounded domain in the plane, and $\phi$ is analytic 
in $G$.
Consider the Bergman space $A^2(G)$.
Let $T_{\phi}:A^2(G) \rightarrow A^2(G)$ be the Toeplitz operator $T_{\phi}(f) = \phi \cdot f$ with symbol $\phi$.
Then we have the lower bound
\[
||[T_{\phi}^*,T_{\phi}] || \ge
\sup_{\substack{W \subset G_{\phi}\\ \phi^{-1}(W) \text{\rm\ piecewise $C^1$}}} 
\frac{\rho_W}{\text{Area}(G_\phi)},
\]
where $\rho_W$ is the torsional rigidity of $W$.  
In particular, if $G$ has smooth boundary, we have 
\[
||[T_{\phi}^*,T_{\phi}] || \ge
\frac{\rho_{G_\phi}}{\text{Area}(G_\phi)}.
\]
}

We also prove a related theorem involving the \emph{torsional rigidity}.
The torsional rigidity of a plane domain $W$ is the supremum of the quotient
$$ \frac{4 \left( \int_W \psi dA \right)^2}{\int_W |\nabla \psi|^2 dA} , $$
taken over all functions $\psi$ continuously differentiable in $W$ and vanishing on the boundary.
From elasticity theory, for a beam of cross section $W$, the torsional rigidity of $W$ is a measure of the beam's resistance to twisting.

\begin{thm}\label{thm:torsion}
\thmtorsiontext
\end{thm}

If we combine the above theorem with Putnam's inequality, using 
$\phi(z)=z$, we obtain the bound 
\begin{equation}
 \label{eq:SV}
\rho \leq \frac{\text{Area}^2(G)}{\pi}.
\end{equation}

Saint-Venant's inequality, which is sharp, states that \cite[p. 121]{PolyaSzego}
\[
\rho \leq \frac{\text{Area}^2(G)}{2\pi}.
\]
Thus, our result combined with Putnam's inequality recovers Saint-Venant's inequality, except for a non-sharp constant.

\noindent {\bf Added in press:} It turns out that the deficiency in recovering the precise form of Saint-Venant's inequality is entirely due to Putnam's inequality not being sharp in
this setting. 
This was our conjecture (see the first concluding remark of this paper)
that was recently proved by J-F. Olsen and M. C. Reguera \cite[Thm. 1, Cor. 1]{OlsenReguera2013}.
Their improvement of Putnam's inequality by a factor of 2 within this setting, along
with our Theorem 1.2, provides a new proof of the Saint-Venant inequality (with sharp constants).

\noindent {\bf Acknowledgement:} We would like to thank Dmitry Khavinson for suggesting to us the problem
of obtaining a lower bound in the context of Bergman space and for valuable suggestions.
We would also like to thank Tom Carroll for helpful suggestions and for informing us of Saint-Venant's inequality.

\section{Multiplication by $z$ and Double Quadrature Domains}\label{sec:dqd}

Let us revisit the case of Smirnov spaces discussed in the introduction.
We will fix our attention on the operator $T_z$ of multiplication by $\phi = z$.
It has already been noticed that quadrature domains give rise to self-commutators of finite rank (\cite{McCarthyYang95}, \cite{McCarthyYang97}, and \cite{ConYan98}).
In this section, we will use this fact to study a specific example.
Actually, we take our domain to be a so-called \emph{double quadrature domain}
which further simplifies computations associated with $[T_z^*,T_z]$ and its norm.

Recall that a bounded domain $G \subset \CC$ is called an \emph{area quadrature domain} if it admits a formula expressing the area integral of 
any function $f$ analytic and integrable in $G$ as a finite sum of weighted point evaluations of the function and its derivatives.
i.e. 
\begin{equation}\label{eq:QF1}
 \int_{\Omega} {f dA} = \sum_{m=1}^{N} \sum_{k=0}^{n_m}{a_{m,k}f^{(k)}(z_m)},
\end{equation}
where $z_m$ are distinct points in $\Omega$ and $a_{m,k}$ are constants (possibly complex) independent of $f$.
In order to define \emph{arclength quadrature domains}, such a formula is prescribed to hold for integration over the boundary with respect to arclength instead of area measure.
A domain that is in both classes is called a \emph{double quadrature domain}.
We emphasize that these domains are very special, since a finite formula is prescribed to hold for an infinite dimensional space of functions.
On the other hand, they are not too special, since even double quadrature domains can be used to approximate any reasonable domain \cite{Bell2012}.
The boundary of an area quadrature domain always has a Schwarz function $S(z)$, i.e., a function complex analytic near $\p G$ that coincides with $\bar{z}$ on $\p G$.
Furthermore, $S(z)$ extends meromorphically into $G$ if $G$ is an area quadrature domain.
When $G$ is a double quadrature domain, 
we can describe a procedure for writing down the entries in a finite matrix representation for $[T_z^*,T_z]$.

As observed in \cite{McCarthyYang95}, the commutator $[T_z^*,T_z]$ has finite rank.
Let us review the details of this fact.
Choose a polynomial $t(z)$ with zeros that cancel the poles of the Schwarz function $S(z)$ in $\ol{G}$,
and decompose each polynomial 
$$p(z) = t(z) q(z) + r(z)$$
using polynomials $q(z)$ and $r(z)$, where the degree of the remainder $r(z)$ is less than the degree of $t(z)$.
Then on $\p G$, the Schwarz function (by definition) matches $\bar{z}$, 
and therefore $\bar{z} t(z)$ matches an analytic function.
This implies that $t(z)q(z)$ is in the kernel of $[T_z^*,T_z]$.
To see this, 
first recall that 
$$T^*_{z} g = P_{E^2(G)}( \bar{z} g ),$$
where $P_{E^2(G)}$ denotes projection onto $E^2(G)$.
Indeed, for any $g,h \in E^2(G)$, we have 
$$ \la T_{z}^* g ,  h \ra =  \la g , T_{z} h \ra = \la g , z h \ra = \la \bar{z} g , h \ra = \la P_{E^2(G)}( \bar{z} g) , h \ra .$$
Since, $\bar{z} t(z)$ matches an analytic function on $\partial G$,
\begin{align*}
 [T_z^*,T_z] t(z) q(z) &= T_z^*T_z t(z) q(z) - T_z T_z^* t(z) q(z) \\
  &= P_{E^2(G)}( \bar{z} z t(z) q(z) ) - z P_{E^2(G)}( \bar{z} t(z) q(z) ) \\
  &= 0 .\\
\end{align*}
Thus, we only need to consider $[T_z^*,T_z]$ acting on the finite dimensional space of remainders which we denote by $R$.
Since $G$ is a double quadrature domain, for each polynomial $r(z) \in R$,
 $[T_z^*,T_z] r(z)$ is a rational function in $z$ and $\bar{z}$ (see \cite[Section 6]{Bell2012}).
Proceed by orthonormalizing $R$ and its image under $[T_z^*,T_z]$.
Note that any inner products arising in the Gram-Schmidt process reduce to integrals that can be calculated using residues
since $G$ is an arclength quadrature domain (again, see \cite[Section 6]{Bell2012}).
Having found an orthonormal basis for $R$ and its image, the entries in the matrix can be calculated by representing
the image of each basis element in terms of the orthonormal basis.

Let us illustrate this with a specific example.
We will then compare the exact value of the norm to the upper and lower estimates mentioned in the introduction.

\noindent {\bf Example:} The following is a simple non-trivial double quadrature domain.
Let $\e >0$ be a real parameter, 
and take the domain $G$ that is the image of the unit disc under the polynomial conformal map $F(w) = w + \e w^2 + \e^2 w^3 /3$ 
(in order for $F$ to be univalent, $\e$ must be sufficiently small).
The Schwarz function of $G$ is given by the formula \cite{Davis}
$$S(z) = F^* \left( \frac{1}{F^{-1}(z)} \right),$$
where $F^*$ denotes conjugation of the coefficients of $F$ (which does nothing in this case since the coefficients of $F$ are real).
From this formula, it can be verified (since the order of the pole of $F^*(1/w)$ is not increased by composition with the conformal map $F^{-1}$) 
that $S(z) = \frac{A}{z^3} + \frac{B}{z^2} + \frac{C}{z} + h(z) ,$
with $h(z)$ analytic in $G$.
Using $t(z) = z^3$ to cancel the pole of $S(z)$, each polynomial can be factored as $p(z) = t(z) q(z) + r(z)$ 
with remainder $r(z)$ of at most second-degree.

For this example we can utilize the explicit conformal map $F$ to further simplify the process described above;
we work in the unit disc using the isometric isomorphism $\Lambda : E^2(G) \rightarrow E^2(\DD) $ induced by $F$.
Namely, $\Lambda$ is defined by $\Lambda \phi = \sqrt{F'} \cdot ( \phi \circ F )$ (see \cite[Chapter 12]{Bellbook} for details).
Notice that for polynomials $r(z)$ of at most second degree, 
$$\Lambda r = \sqrt{F'} \cdot (r \circ F) = (1+ \e w) \cdot r(w+ \e w^2 + \e^2 w^3 / 3 )$$ 
is a polynomial of degree at most $7$.
Thus, Span$(1,w,w^2,..,w^7)$ contains $\Lambda R$,
and the norm of $[T^*,T]$ coincides with the norm of the induced action of $[T^*,T]$ on Span$(1,w,w^2,..,w^7)$.
Moreover, in the disc the monomials $w^k$ give an orthonormal basis.

In order to describe this induced action, 
suppose $\Lambda r(z) = w^k$ and apply $\Lambda$ to $[T_z^*,T_z] r$:
\begin{align*}
 \Lambda [T_z^*,T_z] r &= \Lambda P_{E^2(G)}(z \bar{z} r ) - \Lambda ( z P_{E^2(G)}(\bar{z} r ) ) \\
  &= \Lambda P_{E^2(G)}(z \bar{z} r ) - \sqrt{F'} F [ P_{E^2(G)}(\bar{z} r ) ] \circ F  \\
  &= \Lambda P_{E^2(G)}(z \bar{z} r ) -  F \cdot \Lambda P_{E^2(G)}( \bar{z} r ) \\
  &= P_{E^2(\DD)} ( \Lambda z \bar{z} r ) - F \cdot P_{E^2(\DD)}( \Lambda( \bar{z} r ) ) \\
  &= P_{E^2(\DD)} ( F \bar{F} \Lambda r ) - F \cdot P_{E^2(\DD)}( \bar{F} \Lambda r ) \\
  &= P_{E^2(\DD)} ( F(w) \ol{F(w)} w^k ) - F(w) \cdot P_{E^2(\DD)}( \ol{F(w)} w^k ) , \\
\end{align*}
which is also a polynomial.
Above, we have used the fact that $P_{E^2(\DD)} \Lambda = \Lambda P_{E^2(G)}$ \cite[Theorem 12.3]{Bellbook}.

Next we write down a matrix $M$ representing this action on Span$(1,w,w^2,..,w^7)$.
First note that this matrix is in fact $3 \times 3$.
Indeed, if $\Lambda r = w^k$ for an integer $k>2$, then by a direct calculation using $\bar{w} = 1 / w$ on $\p \DD$
$$\Lambda [T_z^*,T_z] r = P_{E^2(\DD)} ( F(w) \ol{F(w)} w^k ) - F(w) \cdot P_{E^2(\DD)}( \ol{F(w)} w^k ) = 0.$$
Moreover, for $\Lambda r = w^k$ with $k = 0,1,2$ it will turn out that the degree of $\Lambda [T_z^*,T_z] r$ is at most two,
so that we can expand it also in terms of the basis $\{1,w,w^2 \}$.
Let us now calculate the entries in the matrix $M$:

$$M = \left[
\begin{array}{ccc} 	
1 + \e^2 + \frac{\e^4}{9} & \e + \frac{\e^2}{3} & \frac{\e^2}{3} \\
\e + \frac{\e^2}{3} & \e^2 + \frac{\e^4}{9} & \frac{\e^3}{3} \\
\frac{\e^2}{3} & \frac{\e^3}{3} & \frac{\e^4}{9} \\
\end{array}
\right] .$$
The first, second, and third columns are the coefficients in the basis $\{1,w,w^2 \}$ of $\Lambda [T_z^*,T_z] r $, where $\Lambda r$ is $1$, $w$, and $w^2$ respectively.

For the convenience of the reader, we give the details of the computation for the third row.  Thus taking $\Lambda r = w^2$, we obtain

\begin{align*}
 \Lambda [T_z^*,T_z] r &= P_{E^2(\DD)}(F(w) \cdot \ol{F(w)} \cdot w^2) - F(w) \cdot P_{E^2(\DD)}(\ol{F(w)}\cdot w^2) \\
  &=  P_{E^2(\DD)}(F(w) \cdot (w + \e + \frac{\e^2}{3}\bar{w})) - F(w) \cdot P_{E^2(\DD)}(w + \e + \frac{\e^2}{3} \bar{w}) \\
  &= P_{E^2(\DD)}(F(w) \cdot (w + \e)) +P_{E^2(\DD)}(F(w) \cdot \frac{\e^2}{3}\bar{w}) - F(w) \cdot (w + \e ) \\
  &= P_{E^2(\DD)}(F(w) \cdot \frac{\e^2}{3} \bar{w}) = \frac{\e^2}{3} + \frac{\e^3}{3}w + \frac{\e^4}{9}w^2 . \\
\end{align*}

Extracting the coefficients, we obtain the third column in the matrix $M$.

Since $M$ is symmetric, the norm of $M$ is the absolute value of the largest eigenvalue. 
Solving the characteristic polynomial perturbatively in $\e$, we obtain the expansion in $\e$ of the norm 
$$||[T_z^*,T_z]|| = 1 + 2 \e^2 - \frac{1}{9}\e^4 + \frac{2}{3}\e^6 + O(\e^8) .$$

Let us compare this to the upper and lower bounds for $|| [T_z^*,T_z] ||$ provided by Putnam and Khavinson (respectively).
To this end, we calculate the area and perimeter of $G$.
$$\text{Area}(G) = \int_{\DD} F'(w) \cdot \ol{F'(w)} dA(w) = \int_{\DD} (1+\e w)^2(1+\e \bar{w})^2 dA(w).$$
Expanding the integrand and discarding terms that integrate to zero,
$$\int_{\DD} (1+\e w)^2(1+\e \bar{w})^2 dA(w) = \int_{\DD} 1 + 4 \e^2|w|^2+ \e^4 |w|^4  dA(w).$$
Using polar coordinates to calculate this last integral we get
$$\text{Area}(G) = \pi(1+ 2\e^2+\frac{\e^4}{3}).$$

Calculating also the perimeter,
$$P(G) = \int_{\p \DD} |F'(w)| ds = \int_{\p \DD} (1+\e w)(1+\e\bar{w}) ds =  \int_{\p \DD} 1 + \e^2 |w|^2 ds = 2 \pi (1+ \e^2).$$

Since $\phi = z$, $|| \phi' ||_{2}^2$ is the perimeter of $G$, and the lower bound (\ref{eq:khav}) mentioned in the introduction becomes 
$$||[T_z^*,T_z]|| \geq  \frac{4 \text{Area}(G)}{P(G)^2} = \frac{4 (1+2\e^2+\frac{\e^4}{3})^2}{(2 \pi (1+\e^2))^2} = \left(1+ \frac{\e^2}{3} \left(1+\frac{2}{1+\e^2} \right) \right)^2.$$

Combining this with Putnam's inequality (\ref{eq:putnam}), we have

$$ \left(1+ \frac{\e^2}{3} \left(1+\frac{2}{1+\e^2} \right) \right)^2 \leq ||[T_z^*,T_z]|| \leq (1+2\e^2+\frac{\e^4}{3}) .$$

\begin{figure}[h]
    \begin{center}
    \includegraphics[scale=.3]{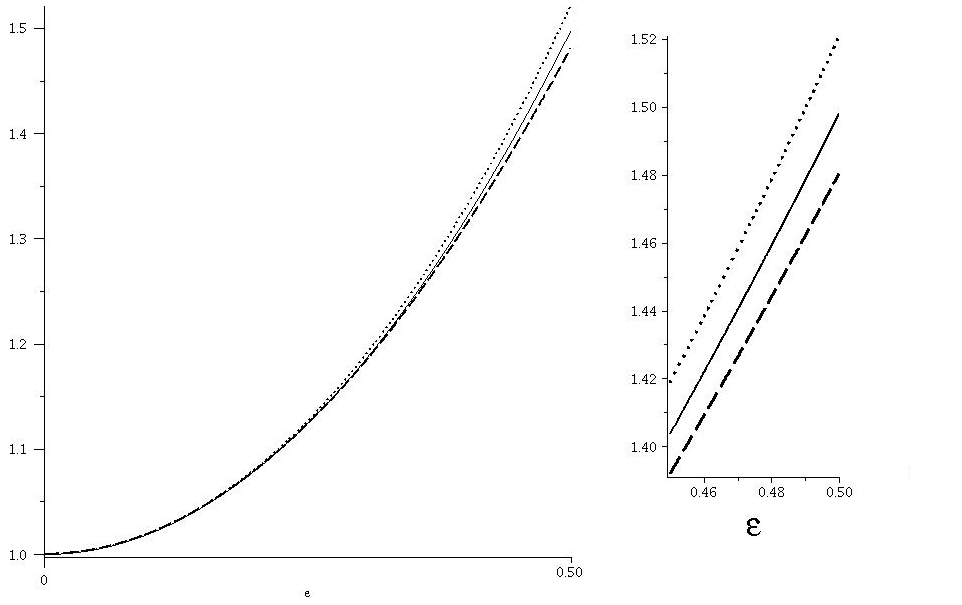}
    \end{center}
    \caption{The norm of $[T_z^*,T_z]$ (solid), and the lower (dashed) and upper (dotted) bounds plotted against $\e$.}
    \label{fig:Norm}
\end{figure}

\begin{figure}[h]
    \begin{center}
    \includegraphics[scale=.25]{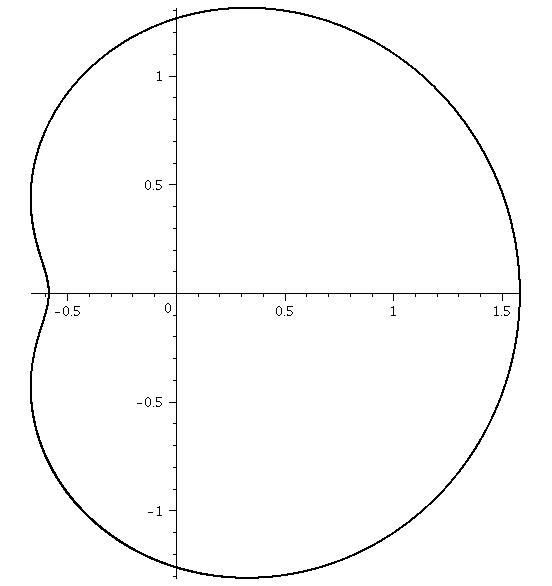}
    \end{center}
    \caption{The domain $G$ when $\e = 0.5$.}
    \label{fig:dqd}
\end{figure}

Figure \ref{fig:Norm} shows a plot of $||[T_z^*,T_z]||$ with respect to $\e$ along with the upper and lower bounds,
which provide a remarkably good estimate even as the geometry of $G$ deviates from a disc (see Figure \ref{fig:dqd}).

In general, even for quadrature domains, the upper and lower bounds need not exhibit such accurate estimates as in this case.
Our main message in this section has been to promote quadrature domains in computational studies as ideal for obtaining explicit results.

\section{A lower bound for the commutator in the case of Bergman space}\label{sec:berg}

Consider the Riemann surface $G_\phi$ formed by sheets over $\phi(G)$, such
that $\phi^{-1}$ can be defined to be one-to-one on the Riemann surface.
More formally, the Riemann surface in question is just $G$ with the
metric $ds = |\phi'(z)| |dz|$.  An arc $\gamma$ on this Riemann
surface has length equal to the length of $\phi(\gamma)$, and a region 
$U$ has area equal to the area of $\phi(U)$ (where overlaps are
counted multiple times).

In this section, we prove Theorem \ref{thm:lambda},
and subsequently give alternative lower estimates.
Recall the statement of the Theorem.
\begin{thmlambda}
\thmlambdatext
\end{thmlambda}

\begin{proof}[Proof of Theorem \ref{thm:lambda}]

We will first consider the case where $\phi$ is univalent, since this is somewhat simpler.  
Then $G_\phi$ is just $\phi(G)$.  
Let $W \subset \phi(G)$.  Let $V = \phi^{-1}(W)$.

Let $P_{A^2(G)}:L^2(G) \rightarrow A^2(G)$ be the Bergman projection. 
The first part of the proof relies on standard ideas from functional analysis 
(cf. \cite{Khav} and \cite{AxlerShapiro83}).

Similarly to the previous section, recall that for $\phi$ analytic, $T^*_{\phi} g = P_{A^2(G)}( \bar{\phi} g )$. 
Indeed, for any $g,h \in A^2(G)$, we have 
$$ \la T_{\phi}^* g ,  h \ra =  \la g , T_{\phi} h \ra = \la g , \phi h \ra = \la \bar{\phi} g , h \ra = \la P_{A^2(G)}( \bar{\phi} g) , h \ra .$$

Since $[T_{\phi}^*,T_{\phi}]$ is a positive normal operator on $A^2(G)$ we have \cite[Theorem 12.25]{Rudin}, 
$$||[T_{\phi}^*,T_{\phi}]|| = \sup_{h \in A^2, ||h||_2 = 1} \la [T_{\phi}^*,T_{\phi}] h , h \ra ,$$
and fixing $h \in A^2(G)$ with $||h||_2 = 1$, we have

\begin{align*}
 \la (T_{\phi}^* T_{\phi}-T_{\phi} T_{\phi}^*) h , h \ra &= ||T_{\phi}h||^2 - ||T_{\phi}^*h||^2 \\
 &= ||\phi h||^2 - ||P_{A^2(G)}(\bar{\phi} h)||^2 \\
 &= ||\bar{\phi} h||^2 - ||P_{A^2(G)}(\bar{\phi} h)||^2.
\end{align*}

This last expression is the square of the $L^2$-distance from $\bar{\phi} h$ to the space $A^2(G)$.
Thus, 
$$||[T_{\phi}^*,T_{\phi}] ||  = \sup_{||h||_2 = 1} \{\inf_{f \in A^2} ||\bar{\phi} h - f ||_2\}^2.$$

Taking $h = \frac{\phi'}{\|\phi' \|_2}$ gives
$$||[T_{\phi}^*,T_{\phi}] ||  \geq  \{\inf_{f \in A^2} ||\bar{\phi}\phi' - f ||_2\}^2 
\frac{1}{\| \phi' \|_2^2}.$$
Note that $\|\phi'\|_2^2 = \text{Area}(\phi(G))$. 

By duality \cite[Chapter 4]{Rudin},
$$||\bar{\phi}\phi'  - f ||_2 = 
\sup_{||g||_2 = 1} \left|\int_G(\bar{\phi}\phi' -f)\bar{g}\,dA \right|.$$

Since we are taking a supremum, any choice of $g$ gives us a lower bound,
but we want to make a careful choice.
First of all, we want to choose $g$ so that the right-hand-side does not depend on $f$.
This means that we want to choose $g$ so that $\int_G f \bar{g}\, dA = 0$.
Accordingly, we will choose $g = \partial_z \psi / ||\partial_z
\psi||_2$, where $\psi$ is in $W^{1,2}(V)$, vanishes on
$\p V$ and on $G \setminus V$,
and will be specified later.  Here 
$W^{1,2}(V)$ denotes the Sobolev space of functions in 
$L^2$ whose distributional derivatives are also in $L^2$.

Let us check that such a choice is in the orthogonal complement to $A^2(G)$ as desired
(this is the easy direction of Havin's Lemma which also provides a converse \cite{Shapiro92}).

First note that if $\phi$ is a $C^\infty$ function with compact support inside $G\setminus V$, then 
\[
\int_G g \phi\, dA = - \frac{1}{\| \partial_z \psi \|_2} \int_G \psi \partial_z \phi \, dA = 0
\]
 since $\phi$ and $\psi$ have disjoint support.
Thus, $g = 0$ in $G \setminus V$. 
We also have that, if $f$ is analytic,
\[
\int_G f \bar{g}\, dA = \int_V f \bar{g} \, dA = 
\frac{1}{\| \partial_z \psi \|_2} \int_V f \overline{\partial_z \psi} \, dA =
\frac{1}{2i} \frac{1}{\| \partial_z \psi \|_2} \int_{\partial V} f \bar{\psi} \, dz = 0,
\]
where we have used the fact that $V$ is piecewise $C^1$ when 
applying the Cauchy-Green formula.

Then 
\begin{equation}\label{eq:elimf}
 \left|\int_G(\bar{\phi}\phi'-f)\bar{g}dA \right| = \left|\int_G
   \overline{\phi g}\phi'  dA \right|= \frac{1}{||\partial_z \psi ||_2}
 \left|\int_V \phi' \ol{\phi} \ol{\partial_z \psi} dA \right|.
\end{equation}

By Green's Theorem, 
\begin{align*}
\int_V  \phi' \ol{\phi'}\ol{\psi} + \phi'\ol{\phi} \ol{\partial_z
  \psi}\,  dA &= \int_V \phi' \ol{\partial_z (\phi \psi)}\, dA \\
  &= \int_V \phi' \partial_{\ol{z}} (\ol{\phi \psi})\, dA \\
  &= \frac{1}{2i}\int_{\partial V}  \phi' \ol{\phi \psi}\, dz = 0,
\end{align*}
where the last integral is zero because $\psi$ vanishes on $\partial V$.

Thus, $$\int_V \bar{\phi} \phi' \ol{\partial_z \psi}  dA = 
 -\int_V  \phi '\ol{\phi'}\bar{\psi} dA ,$$
and (\ref{eq:elimf}) becomes
\begin{equation}\label{eq:prelim}
\frac{1}{||\partial_z \psi ||_2} \left|\int_V \phi' \ol{\phi'}\bar{\psi} dA \right| = \frac{1}{||\partial_z \psi ||_2} \left|\int_V |\phi'|^2 \psi \, dA \right| .
\end{equation}

Now choose $\psi(z) = \psi_{W}(\phi(z))$, where 
$\psi_W$ denotes the first eigenfunction of the Laplacian for the
domain $W$.  If $\phi(z) \not \in W$, define $\psi(z)=0$. 

Note that 
\[
\| \partial_z \psi_{W}(\phi(z)) \|_2^2 = 
\int_V |\{[\partial_z \psi_{W}](\phi(z))\} \phi'(z)|^2 dA = 
\int_{W} |\partial_z \psi_{W}(w)|^2 dA,
\]
using the chain rule and a change of variables.  By the definition of $\psi_W$, this last 
integral is finite, so
this calculation shows that $\partial_z \psi$ is in 
$L^2$.  Since $\psi$ is real valued, this means $\partial_x \psi$ and 
$\partial_y \psi$ are in $L^2,$ so by the Poincar\'{e} inequality
(see \cite{Evans}),
 $\psi$ is in 
$L^2$.  Thus, $\psi$ is in $W^{1,2}$, as claimed, and the integrals in our 
calculation are well defined.  

Also, under a change of variables,
\[
\int_G \psi_{W}(\phi(z)) |\phi'(z)|^2 \, dA = \int_{W}
\psi_{W} (w) \, dA. 
\]

So (\ref{eq:prelim}) becomes 
\begin{equation}\label{quotient}
\frac{1}{\|\partial_z \psi_{W}\|_2} \left| \int_{W} \psi_{W} dA \right|.
\end{equation}
(Note, the $L^2$ norm is over $W$).

The reason for this choice is that the first eigenfunction
$\psi_{W}$ 
solves a variational problem, minimizing the Rayleigh quotient,
and $||\partial_z \psi_{W} ||_2 = (1/2) || \nabla \psi_{W}
||_2 = 
\sqrt{\lambda} / 2 || \psi_{W} ||_2$, where $\lambda$ is the first eigenvalue of the Laplacian.

Thus, with this choice,
\[
\frac{1}{\|\partial_z \psi_{W}\|_2} \left| \int_{W}
\psi_{W} dA \right| = 
\frac{2 || \psi_{W} ||_1}{\sqrt{\lambda} ||\psi_{W} ||_2} ,
\]
where we have used the fact that the first eigenfunction $\psi_W$ is non-negative \cite[p. 24]{PolyaSzego}.

Now we use the ``reverse H\"older's inequality'' of Payne and Rayner \cite{PayneRayner} (see also \cite{CarrollRatzkin2012}), 
which gives
\[
\frac{ \| \psi_{W} \|_1}{ \| \psi_{W} \|_2} \ge
\frac{2\sqrt{\pi}}{\sqrt{\lambda}}.
\]
Then this gives 
\[
\frac{2 || \psi_{W} ||_1}{\sqrt{\lambda} ||\psi_{W} ||_2} \ge 
\frac{4 \sqrt{\pi}}{\lambda},
\]
which implies that 
\[
||[T_{\phi}^*,T_{\phi}] || \ge 
\frac{4^2 \pi}{\lambda^2 \text{Area}(\phi(G))}
\]
where $\lambda$ is the first eigenvalue of the Laplacian on
$W$. 

If $\phi$ is locally univalent in $G$ the same proof holds, 
with $G_\phi$ in place of $\phi(G)$.  The inequality of Payne and Rayner 
still applies since $G_\phi$ has curvature $0$ everywhere.  (The
inequality does not necessarily apply if $\phi$ is not locally
univalent, since then $G_\phi$ has branch points, which is where we
need the assumption of local univalence). 
\end{proof}

In the proof above, we chose $\psi_W$ to minimize a Rayleigh quotient
$$\frac{|| \nabla \psi_{W}||_2}{|| \psi_{W} ||_2},$$
i.e., to maximize the quotient
$$\frac{|| \psi_{W} ||_2}{|| \nabla \psi_{W}||_2}.$$

We can get a more direct estimate if we instead choose $\psi_W$ to be the \emph{stress function} \cite[p. 24]{PolyaSzego} that maximizes the quotient
$$\left( \frac{ 2 || \psi_{W} ||_1}{|| \nabla \psi_{W}||_2} \right)^2,$$
over continuously differentiable functions $\psi_W$ vanishing on the boundary of $W$.
Then the inequality
\begin{equation*}
||[T_{\phi}^*,T_{\phi}] || \ge 
\left( \frac{ 2 \left|\int_W \psi_W \, dA \right|}{|| \nabla \psi_W||_2} \right)^2 \frac{1}{\text{Area}(\phi(G))}
\end{equation*}
is still true by the same argument as in the above proof.
Since $\psi_W$ is superharmonic and has vanishing boundary values \cite[p. 24]{PolyaSzego},
it is non-negative.
This implies $ \left|\int_W \psi_W \, dA \right| = || \psi_{W} ||_1$,
giving
\begin{equation*}
||[T_{\phi}^*,T_{\phi}] || \ge \left( \frac{ 2 || \psi_{W} ||_1}{|| \nabla \psi_{W}||_2} \right)^2 \frac{1}{\text{Area}(\phi(G))},
\end{equation*}
which can be stated more simply in terms of the torsional rigidity $\rho_W$.

This gives an alternative form of Theorem \ref{thm:lambda}.  Here the requirement that 
$\phi$ be locally univalent is not needed.

\begin{thmtorsion}
\thmtorsiontext
\end{thmtorsion}

\noindent {\bf Note:}  The hypothesis that $\phi$ is locally univalent can be removed from the assumptions in Theorem \ref{thm:lambda} 
if we require that $W$ does not contain any branch points of $\phi^{-1}$.

We now state a purely geometric bound in the case when $\phi$ is univalent.
The radius $R_I$ of the largest disc contained in a domain is called the \emph{inradius} of the domain.

\begin{cor}\label{cor:geo}
 Under the same hypothesis as Theorem \ref{thm:torsion} and with the additional assumption that $\phi$ is univalent in $G$,
we have the lower bound
\[
||[T_{\phi}^*,T_{\phi}] || \ge
\frac{\pi R_I^4}{2 \text{Area}(\phi(G))} 
\]
where $R_I$ is the inradius of $\phi(G)$.  
\end{cor}
The bound 
\[
||[T_{\phi}^*,T_{\phi}] || \ge
\frac{\pi R_I^4}{2 \text{Area}(G_\phi)} 
\]
still holds if we remove the assumption of the univalence of $\phi$, as long
as we require that the discs in the definition of the inradius do not
contain any branch point of $G_\phi$. 

\begin{proof}
Take $W$ to be the largest disc contained in $\phi(G)$.
Then $\rho_W = \frac{\pi R_I^4}{2}$, and applying Theorem \ref{thm:torsion},
\[
||[T_{\phi}^*,T_{\phi}] || \ge
\frac{\pi R_I^4}{2 \text{Area}(\phi(G))}
\]
\end{proof}

\section{Concluding remarks}

{\bf 1.}  The lower bound in Theorem \ref{thm:torsion} seems to be sharp,
since in the case of the unit disc with $\phi = z$
the square of the distance of $\bar{z}$ from $A^2(\DD)$
$$\text{dist}^2(\bar{z}, A^2(\DD)) = 1/2$$
(see \cite[Remark 4.3]{DimaZdenka}).
Since the constant in inequality (\ref{eq:SV}) is off by a factor of $2$ (compared to Saint-Venant's),
we expect that the constant in Putnam's inequality is not sharp in the setting of this paper 
(although it is sharp in general and in the restricted setting of Smirnov space \cite{Khav}).

We conjecture that for a Toeplitz operator with analytic symbol $\phi$
acting on the Bergman space of $G$, the following improved version of Putnam's inequality holds
\[
||[T_{\phi}^*,T_{\phi}] || \leq \frac{\text{Area}(\phi(G))}{2 \pi} .
\]

\noindent {\bf Added in press:} Recently, J-F. Olsen and M. C. Reguera \cite[Thm. 1, Cor. 1]{OlsenReguera2013} proved our conjecture. 
Along with our Theorem 1.2, this provides the following operator theoretic ``isoperimetric sandwich'', in this case recovering Saint-Venant's
isoperimetric inequality (as opposed to the classical isoperimetric inequality appearing in Khavinson's study):
$$ \frac{\rho}{\text{Area}(G)} \leq ||[T_{z}^*,T_{z}] || \leq \frac{\text{Area}(G)}{2 \pi}.$$

{\bf 2.} It would be interesting to extend this study to spaces of analytic functions on domains in $\CC^n$.
Consider the Bergman space $A^2(\Omega)$ of a pseudoconvex domain $\Omega \subset \CC^n$.
In order to illustrate how the ideas might proceed, 
let us work out the example of a Toeplitz operator $T_{\phi}: A^2(\Omega) \rightarrow A^2(\Omega)$
with a linear function $\phi(z) = \sum_{j=1}^n a_j z_j$ as its symbol, where $a_j$ are complex constants.

With the goal of computing a lower bound for the norm of the commutator,
following the first part of the proof of Theorem \ref{thm:lambda} leads to the same formula
$$||[T_{\phi}^*,T_{\phi}] ||  = \sup_{||h||_2 = 1} \{\inf_{f \in A^2} ||\bar{\phi} h - f ||_2\}^2.$$

Taking $h = \frac{1}{\text{Vol}(\Omega)}$ gives
\begin{equation}\label{eq:scv0}
||[T_{\phi}^*,T_{\phi}] ||  \geq  \{\inf_{f \in A^2} ||\bar{\phi} - f ||_2\}^2 
\frac{1}{\text{Vol}(\Omega)}.
\end{equation}

Again by duality,
\begin{equation}\label{eq:scv}
 ||\bar{\phi} - f ||_2 = 
\sup_{||g||_2 = 1} \left|\int_{\Omega}(\bar{\phi}-f)\bar{g}\, dV \right|,
\end{equation}
where $dV = dA(z_1) \wedge ... \wedge dA(z_n)$.

As in the proof of the Theorem, we choose $g$ in the orthogonal complement of the Bergman space.
We are guided by Rosay's Lemma (a version of Havin's Lemma in $\CC^n$, see \cite{Bell93})
describing the orthogonal complement of $A^2(\Omega)$.
Namely, functions in $A^2(\Omega)^{\perp}$ are of the form $\vartheta \alpha$,
where $\alpha = \sum_{j=1}^{n} \alpha_j d\bar{z_j}$ is a $(0,1)$-form
of $C^{\infty}$ functions $\alpha_j$ vanishing on the boundary of $\Omega$,
and $\vartheta$ is the formal adjoint of the $\bar{\p}$-operator,
defined on $(0,1)$-forms by $\vartheta \alpha$ = $\sum_{j=1}^n\frac{\p \alpha_j}{\p z_j}$  (a complex ``divergence'' operator).

Choosing $g = \frac{\vartheta \alpha}{||\vartheta \alpha||_2}$ as indicated above, (\ref{eq:scv}) becomes
\begin{equation}\label{eq:scv2}
 \frac{\left|\int_{\Omega}\overline{\phi \vartheta \alpha} \, dV \right|}{||\vartheta \alpha||_2}.
\end{equation}

Let's take $\alpha_j = \bar{a_j}\psi$, where $\psi \in C^{\infty}(\Omega)$ is real, positive, and vanishes on the boundary of $\Omega$.
Integrating by parts in the numerator of (\ref{eq:scv2}) gives
\begin{equation}\label{eq:scv3}
\sum_{j=1}^n  |a_j|^2 \frac{  ||\psi||_1 }{||\vartheta \alpha||_2}.
\end{equation}

\noindent {\bf Remark:} By analogy with the case of one variable, 
it should be clear that we next want to choose $\psi$ to maximize this quotient (\ref{eq:scv3}).
Perhaps extremal problems like this could give one approach for defining ``torsional rigidity'' for domains in $\CC^n$.

In order to make some concrete calculation, let us further specialize and take for example the polydisc $\Omega = \DD^n$ as our domain.
For the purpose of making the quotient (\ref{eq:scv3}) large, 
the choice 
$$\psi = \prod_{j=1}^n (1-|z_j|^2)$$ 
seems to be reasonable (but perhaps not optimal).

In the denominator of (\ref{eq:scv3}), the off-diagonal terms in the product $\vartheta \alpha \overline{\vartheta \alpha }$ integrate to zero.
Reversing the order of integration and summation and moving the constant outside each integral, this leaves
$$||\vartheta \alpha||_2 = \sqrt{ \sum_{j=1}^n |a_j|^2 \int_{\DD^n} \left(  |z_j|^2 \prod_{k \neq j} (1 - |z_k|^2)^2\right)  dV } .$$

Each integral in the sum can be calculated as an iterated integral using polar coordinates over each disc,
and (\ref{eq:scv2}) becomes
$$ \frac{ \sum_{j=1}^n |a_j|^2 \left(\frac{\pi}{2}\right)^n }{ \sqrt{\sum_{j=1}^n |a_j|^2 \frac{\pi^n}{2 \cdot 3^{n-1}}} }.$$
Now, also using $\text{Vol}(\Omega)= \pi^n$, (\ref{eq:scv0}) becomes
$$  ||[T_{\phi}^*,T_{\phi}] || \geq \frac{3^{n-1}}{2^{2n-1}} \sum_{j=1}^n |a_j|^2.$$
To apply Putnam's inequality, we note that the spectrum of $T_{\phi}$ is the disc of radius $\sum_{j=1}^n|a_j|$, so we have
\[
||[T_{\phi}^*,T_{\phi}] || \leq \frac{(\sum_{j=1}^n|a_j|)^2 \pi}{\pi} = \left( \sum_{j=1}^n|a_j| \right)^2 \leq n  \sum_{j=1}^n|a_j|^2,
\]
where the last inequality is a direct application of the Cauchy-Schwarz inequality to the vectors $(1,1,..,1)$ and $(|a_1|,|a_2|,..,|a_n|)$ .

Thus, the commutator of multiplication by a linear function on the Bergman space of the polydisc can be compared from above and below to
$S_a := \sum_{j=1}^n |a_j|^2,$

$$ \frac{3^{n-1}}{2^{2n-1}} S_a  \leq ||[T_{\phi}^*,T_{\phi}] || \leq n  S_a .$$

{\bf 3.} Some aspects of Section \ref{sec:dqd} resemble \cite{Bolt} (cf. \cite{Dostanic}) 
where the spectrum of the Kerzman-Stein operator was studied in terms of the geometry of the domain.
The ellipse was selected as a simple non-trivial specimen, 
but in light of the above, perhaps a double quadrature domain would give more explicit results.

{\bf 4.} Throughout, we have considered only Toeplitz operators with analytic symbol.
Some of the ideas in Section \ref{sec:dqd} can be extended to the operator 
$T_H$ defined by $T_H f = P (H \cdot f)$, where $H(z)$ is the complex unit tangent vector of $\partial G$.
(In general, $H(z)$ does not have an analytic extension to all of $G$.)
Like the operator $T_z$ (multiplication by $z$), $T_H$ reduces to the shift operator when $G$ is a disc.

For the commutator of $T_H$, 
we have 
$$[T_H^*,T_H] f = P_{E^2(G)}( \bar{H} P_{E^2(G)}( H f )) - P_{E^2(G)}( H P_{E^2(G)}(\bar{H} f ) ).$$
Suppose $G$ is an arclength quadrature domain.
Then $H$, and hence $\bar{H} = 1 / H$, has a meromorphic extension to $G$.
Choose a polynomial $t(z)$ that cancels the poles of both $H$ and $\bar{H}$, 
and let $q$ be any polynomial.
Then 
$$P_{E^2(G)}( \bar{H} P_{E^2(G)}( H t q ) )= P_{E^2(G)}( \bar{H} H t q ) = \bar{H} H t q ,$$ 
and similarly 
$$P_{E^2(G)}( H P_{E^2(G)}(\bar{H} t q ) ) = \bar{H} H t q ,$$ 
so the commutator vanishes.
Decomposing each polynomial with respect to $t(z)$ as in Section \ref{sec:dqd},
we see that $[T_H^*,T_H]$ has finite rank.

\bibliographystyle{amsplain}

\end{document}